\documentclass{amsart}%
\usepackage{amsfonts}
\usepackage{amsmath}
\usepackage{amssymb}
\usepackage{graphicx}
\usepackage{hyperref}%
\providecommand{\U}[1]{\protect\rule{.1in}{.1in}}
\newtheorem{theorem}{Theorem}
\theoremstyle{plain}

\newtheorem{corollary}{Corollary}

\newtheorem{definition}{Definition}

\newtheorem{proposition}{Proposition}
\newtheorem{remark}{Remark}

\numberwithin{equation}{section}
\begin{document}
\title[{\normalsize Generalization of 2-absorbing quasi primary ideals}]{{\normalsize Generalization of 2-absorbing quasi primary ideals}}
\author{Emel Aslankarayi\u{g}it U\u{g}urlu}
\address{Department of Mathematics, Marmara University, Istanbul, Turkey.}
\email{emel.aslankarayigit@marmara.edu.tr}
\author{\"{U}nsal Tekir}
\curraddr{Department of Mathematics, Marmara University, Istanbul, Turkey.}
\email{utekir@marmara.edu.tr}
\author{Suat Ko\c{c}}
\address{Department of Mathematics, Marmara University, Istanbul, Turkey.}
\email{suat.koc@marmara.edu.tr}

\begin{abstract}
In this article, we introduce and study the concept of $\phi$-2-absorbing
quasi primary ideals in commutative rings. Let $R\ $be a commutative ring with
a nonzero identity and $L(R)$ be the lattice of all ideals of $R.\ $Suppose
that $\phi:L(R)\rightarrow L(R)\cup\left\{  \emptyset\right\}  $ is a
function. A proper ideal $I\ $of $R\ $is called a $\phi$-2-absorbing quasi
primary ideal of $R\ $if $a,b,c\in R$ and whenever $abc\in I-\phi(I),$ then
either $ab\in\sqrt{I}\ $or $ac\in\sqrt{I}\ $or $bc\in\sqrt{I}.\ $In addition
to giving many properties of $\phi$-2-absorbing quasi primary ideals, we also
use them to characterize von Neumann regular rings.

\bigskip

Keywords. weakly 2-absorbing quasi primary ideal, $\phi$-2-absorbing quasi primary
ideal, von Neumann regular ring.

\bigskip

2010 Mathematics Subject Classification. 13A15, 16E50

\end{abstract}
\maketitle

\section{\textbf{Introduction}}

In this article, we focus only on commutative rings with nonzero identity and
nonzero unital modules. Let $R$ always denote such a ring and $M\ $denote such
an $R$-module. $L(R)$ denotes the lattice of all ideals of $R.\ $Let $I$ be a
proper ideal of $R$, the set $\{r\in R$ $|$ $rs\in I$ for some $s\in
R\diagdown I\}$ will be denoted by $Z_{I}(R)$. Also the radical of $I$ is
defined as $\sqrt{I}:=\{r\in R$ $|$ $r^{k}\in I$ for some $k\in%
\mathbb{N}
\}$ and for $x\in R,$ $(I:x)$ denote the ideal $\{r\in R$ $|$ $rx\in I\}$ of
$R$. A proper ideal $I$ of a commutative ring $R$ is \textit{prime }if
whenever $a_{1},a_{2}\in R$ with $a_{1}a_{2}\in I$, then $a_{1}\in I$ or
$a_{2}\in I,$ \cite{atiyah}. In 2003, the authors \cite{AS} said that if
whenever $a_{1},a_{2}\in R$ with $0_{R}\neq a_{1}a_{2}\in I$, then $a_{1}\in
I$ or $a_{2}\in I,$ a proper ideal $I$ of a commutative ring $R$ is
\textit{weakly prime}. In \cite{BS}, Bhatwadekar and Sharma defined a proper
ideal $I$ of an integral domain $R$ as \textit{almost prime (resp. $n$-almost
prime)} if for $a_{1},a_{2}\in R$ with $a_{1}a_{2}\in I-I^{2}$, (resp.
$a_{1}a_{2}\in I-I^{n},n\geq3$) then $a_{1}\in I$ or $a_{2}\in I$. This
definition can be made for any commutative ring $R$. Later, Anderson and
Batanieh \cite{AB} introduced a concept which covers all the previous
definitions in a commutative ring $R$ as following: Let $\phi:L(R)\rightarrow
L(R)\cup\{\emptyset\}$ be a function, where $L(R)$ denotes the set of all
ideals of $R$. A proper ideal $I$ of a commutative ring $R$ is called $\phi
$-\textit{prime} if for $a_{1},a_{2}\in R$ with $a_{1}a_{2}\in I-\phi(I)$,
then $a_{1}\in I$ or $a_{2}\in I.$ They defined the map $\phi_{\alpha
}:L(R)\rightarrow L(R)\cup\{\emptyset\}$ as follows:

\begin{itemize}
\item[(i)] $\phi_{\emptyset}$ : $\phi(I)=\emptyset$ defines prime ideals.

\item[(ii)] $\phi_{0}$ : $\phi(I)=\{0_{R}\}$ defines weakly prime ideals.

\item[(iii)] $\phi_{2}$ : $\phi(I)=I^{2}$ defines almost prime ideals.

\item[(iv)] $\phi_{n}$ : $\phi(I)=I^{n}$ defines $n$-almost prime
ideals$(n\geq2).$

\item[(v)] $\phi_{\omega}$ : $\phi(I)=\cap_{n=1}^{\infty}I^{n}$ defines
$\omega$-prime ideals.

\item[(vi)] $\phi_{1}$ : $\phi(I)=I$ defines any ideal.
\end{itemize}

\bigskip

\bigskip The notion of 2-absorbing ideal, which is a generalization of prime
ideal, was introduced by Badawi as the following: a proper ideal $I$ of $R$ is
called a 2-absorbing ideal of $R$ if whenever $a,b,c\in R$ and $abc\in I$,
then $ab\in I$ or $ac\in I$ or $bc\in I$, see \cite{Ba}. Also, the notion is
investigated in \cite{AnBa}, \cite{PaBa}, \cite{DaPu}, \cite{Da} and
\cite{DaSo}. Then, the notion of 2-absorbing primary ideal, which is a
generalization of primary ideal, was introduced in \cite{BaTeYe} as: a proper
ideal $I$ of $R$ is called a 2-absorbing primary ideal of $R$ if whenever
$a,b,c\in R$ and $abc\in I$, then $ab\in I$ or $ac\in\sqrt{I}$ or $bc\in
\sqrt{I}$. Note that a 2-absorbing ideal of a commutative ring $R$ is a
2-absorbing primary ideal of $R$. But the converse is not true. For example,
consider the ideal $I=$ $(20)$ of $%
\mathbb{Z}
$. Since $2.2.5\in I$, but $2.2\notin I$ and $2.5\notin I$, $I$ is not a
2-absorbing ideal of $%
\mathbb{Z}
$. However, it is clear that $I$ is a 2-absorbing primary ideal of $%
\mathbb{Z}
$, since $2.5\in\sqrt{I}$. In 2016, the authors introduced the concept of
$\phi$-2-absorbing primary ideal which a proper ideal $I$ of $R$ is called a
$\phi$-2-absorbing primary ideal of $R$ if whenever $a,b,c\in R$ and $abc\in
I-\phi(I)$, then $ab\in I$ or $ac\in\sqrt{I}$ or $bc\in\sqrt{I}$, see
\cite{turkish}.

\bigskip

On the other hand, the concept of quasi primary ideals in commutative rings
was introduced and investigated by Fuchs in \cite{Fu}. The author called an
ideal $I$ of $R$ as a quasi primary ideal if $\sqrt{I}$ is a prime ideal. In
\cite{qua}, the notion of 2-absorbing quasi primary ideal is introduced as
following: a proper ideal $I$ of $R$ to be 2-absorbing quasi primary if
$\sqrt{I}$ is a 2-absorbing ideal of $R$.

\bigskip

In this paper, our aim to obtain the generalizations of the concept of the
quasi primary ideal and 2-absorbing quasi primary ideal. For this, firstly we
define the $\phi$-quasi primary ideal. Let $\phi:L(R)\rightarrow
L(R)\cup\left\{  \emptyset\right\}  $ be a function and $I\ $\ be a proper
ideal of $R.$ Then $I$ is said to be a $\phi$-quasi primary ideal if whenever
$a,b\in R$ and $ab\in I-\phi(I)$, then $a\in\sqrt{I}$ or $b\in\sqrt{I}$.
Similarly, $I$ is called a $\phi$-2-absorbing quasi primary ideal if whenever
$a,b,c\in R$ and $abc\in I-\phi(I)$, then $ab\in\sqrt{I}$ or $ac\in\sqrt{I}$
or $bc\in\sqrt{I}$. In Section 2, firstly we investigate the basic properties
of a $\phi$-quasi primary ideal and a $\phi$-2-absorbing quasi primary. By the
help of Theorem \ref{t1} and Theorem \ref{t2}, we give a diagram which
clarifies the place of $\phi$-2-absorbing quasi primary ideal in the lattice
of all ideals $L(R)\ $of $R,$ see Figure 1. In Proposition \ref{method}, we
give a method for constructing $\phi$-2-absorbing quasi primary ideal in
commutative rings. Also, if $\phi(I)$ is a quasi primary ideal of $R,$ we
prove that $I\ $is a $\phi$-2-absorbing quasi primary ideal of
$R\Leftrightarrow$ $I\ $is a 2-absorbing quasi primary ideal of $R,$ see
Corollary \ref{equivalent}. With Theorem \ref{nakayama}, we obtain the
Nakayama%
\'{}%
s Lemma for weakly (2-absorbing) quasi primary ideals. Moreover, we examine
the notion of "$\phi$-2-absorbing quasi primary ideal" in $S^{-1}R,$ where $S$
is a multiplicatively subset of $R$. In Theorem \ref{idealization}, we
characterize the weakly 2-absorbing quasi primary ideal of $R\propto M,$ that
is, the trivial ring extension, where $M\ $is an $R$-module. In Theorem
\ref{VonNue}, we describe von Neumann regular rings in terms of $\phi
$-2-absorbing quasi primary ideals. Finally, with the all results of the
Section 3, we characterize $\phi$-2-absorbing quasi primary ideal in direct
product of finitely many commutative rings.

\bigskip

\bigskip

\bigskip

\bigskip

\section{\textbf{Characterization of }$\phi$\textbf{-2-absorbing quasi primary
ideals}}

Throughout the paper, $\phi:L(R)\rightarrow L(R)\cup\left\{  \emptyset
\right\}  $ always denotes a function.

\begin{definition}
Let $R\ $be a ring and $I\ $a proper ideal of $R.\ $

(i)\ $I$ is said to be a $\phi$-quasi primary ideal if whenever $a,b\in R$ and
$ab\in I-\phi(I)$, then $a\in\sqrt{I}$ or $b\in\sqrt{I}$.

(ii) $I$ is said to be a $\phi$-2-absorbing quasi primary ideal if whenever
$a,b,c\in R$ and $abc\in I-\phi(I)$, then $ab\in\sqrt{I}$ or $ac\in\sqrt{I}$
or $bc\in\sqrt{I}$.
\end{definition}

\begin{definition}
Let $\phi_{\alpha}:L(R)\rightarrow L(R)\cup\{\emptyset\}$ be one of the
following special functions and $I\ $a $\phi_{\alpha}$-quasi primary
($\phi_{\alpha}$-2-absorbing quasi primary) ideal of $R.\ $Then,%
\begin{align*}
\phi_{\emptyset}(I)  &  =\emptyset\ \ \ \ \ \text{quasi primary (2-absorbing
quasi primary) ideal}\\
\phi_{0}(I)  &  =0_{R}\ \ \ \text{weakly quasi primary (weakly 2-absorbing
quasi primary) ideal}\\
\phi_{2}(I)  &  =I^{2}\ \ \ \text{almost quasi primary (almost 2-absorbing
quasi primary) ideal}\\
\phi_{n}(I)  &  =I^{n}\ \ \ n\text{-almost quasi primary (}n\text{-almost
2-absorbing quasi primary) ideal\ }(n\geq2)\\
\phi_{\omega}(I)  &  =\cap_{n=1}^{\infty}I^{n}\ \ \ \omega\text{-quasi primary
(}\omega\text{-2-absorbing quasi primary) ideal}\\
\phi_{1}(I)  &  =I\ \ \ \text{any ideal.}%
\end{align*}

\end{definition}

Note that since $I-\phi(I)=I-(I\cap\phi(I))$, for any ideal $I$ of $R$,
without loss of generality, assume that $\phi(I)\subseteq I.$ Let $\psi_{1},$
$\psi_{2}:L(R)\rightarrow L(R)\cup\{\emptyset\}$ be two functions, if
$\psi_{1}(I)\subseteq\psi_{2}(I)$ for each $I\in L(R),$ we denote $\psi
_{1}\leq\psi_{2}.$ Thus clearly, we have the following order: $\phi
_{\emptyset}\leq\phi_{0}\leq\phi_{\omega}\leq\cdots\leq\phi_{n+1}\leq\phi
_{n}\leq\cdots\leq\phi_{2}\leq\phi_{1}$.\ Also $2$-almost quasi primary
($2$-almost 2-absorbing quasi primary) ideals are exactly almost quasi primary
(almost 2-absorbing quasi primary) ideals.

\begin{proposition}
\label{pp}Let $R$\ be a ring and $I$ be a proper ideal $R.$ Let $\psi_{1}%
,\psi_{2}:L(R)\rightarrow L(R)\cup\{\emptyset\}$ be two functions with
$\psi_{1}\leq\psi_{2}.$

(i) If $I$ is a $\psi_{1}$-quasi primary ideal of $R,$ then $I$ is a $\psi
_{2}$-quasi primary ideal of $R.$

(ii)$\ I$ is a quasi primary ideal $\Rightarrow$ $I$ is a weakly quasi primary
ideal $\Rightarrow$ $I$ is an $\omega$-quasi primary ideal $\Rightarrow$ $I$
is an $(n+1)$-almost quasi primary ideal $\Rightarrow$ $I$ is an $n$-almost
quasi primary ideal $(n\geq2)\Rightarrow$ $I$ is an almost quasi primary ideal.

(iii)$\ I$ is an $\omega$-quasi primary ideal if and only if $I$ is an
$n$-almost quasi primary ideal for each $n\geq2.$

(iv) If $I$ is a $\psi_{1}$-2-absorbing quasi primary ideal of $R,$ then $I$
is a $\psi_{2}$-2-absorbing quasi primary ideal of $R.$

(v) $I$ is 2-absorbing quasi primary ideal $\Rightarrow$ $I$ is weakly
2-absorbing quasi primary ideal $\Rightarrow$ $I$ is an $\omega$-2-absorbing
quasi primary ideal $\Rightarrow$ $I$ is an $(n+1)$-almost 2-absorbing quasi
primary ideal $\Rightarrow$ $I$ is an $n$-almost 2-absorbing quasi primary
ideal $(n\geq2)\Rightarrow$ $I$ is an almost 2-absorbing quasi primary ideal.

(vi) $I$ is an $\omega$-2-absorbing quasi primary ideal if and only if $I$ is
an $n$-almost 2-absorbing quasi primary ideal for each $n\geq2.$
\end{proposition}

\begin{proof}
(i): It is evident.

(ii): Follows from (i).

(iii): Every $\omega$-quasi primary ideal is an $n$-almost quasi primary ideal
for each $n\geq2\ $since $\phi_{\omega}\leq\phi_{n}.\ $Now, let $I\ $be an
$n$-almost quasi primary ideal for each $n\geq2.\ $Choose elements $a,b\in R$
such that$\ ab\in I-$ $\cap_{n=1}^{\infty}I^{n}.\ $Then we have $ab\in
I-I^{n}\ $for some $n\geq2.\ $Since $I\ $is an $n$-almost quasi primary ideal
of $R,\ $we conclude either $a\in\sqrt{I}$ or $b\in\sqrt{I}.\ $Therefore,
$I\ $is an $\omega$-quasi primary ideal.

(iv): It is evident.

(v): Follows from (iv).

(vi): Similar to (iii).
\end{proof}

\begin{theorem}
\label{t1}(i) If $\sqrt{I}=I,\ $then $I\ $is a $\phi$-2-absorbing quasi
primary ideal of $R\ $if and only if $I\ $is a $\phi$-2-absorbing ideal of
$R.$

(ii)\ If $I$ is a $\phi$-2-absorbing quasi primary ideal of $R$ and
$\phi(\sqrt{I})=\sqrt{\phi(I)}$, then $\sqrt{I}$ is a $\phi$-2-absorbing ideal
of $R$.

(iii)\ If $\sqrt{I}$ is a $\phi$-2-absorbing ideal of $R$ and $\phi(\sqrt
{I})\subseteq\phi(I),$ then $I$ is a $\phi$-2-absorbing quasi primary ideal of
$R.$

(iv) If $I$ is a $\phi$-quasi primary ideal of $R$ and $\phi(\sqrt{I}%
)=\sqrt{\phi(I)}$, then $\sqrt{I}$ is a $\phi$-prime ideal of $R$.

(v) If $\sqrt{I}$ is a $\phi$-prime ideal of $R$ and $\phi(\sqrt{I}%
)\subseteq\phi(I),$ then $I$ is a $\phi$-quasi primary ideal of $R.$
\end{theorem}

\begin{proof}
(i): It is evident.

(ii):\ Let $I$ be a $\phi$-2-absorbing quasi primary ideal of $R.$ Take
$a,b,c\in R$ such that $abc\in\sqrt{I}-\phi(\sqrt{I})$. Then there exists a
positive integer $n$ such that $(abc)^{n}=a^{n}b^{n}c^{n}\in I.$ Since
$abc\notin\phi(\sqrt{I})$ and $\phi(\sqrt{I})=\sqrt{\phi(I)},$ we get
$abc\notin\sqrt{\phi(I)},$ so $a^{n}b^{n}c^{n}\notin\phi(I).$ Thus, by our
hypothesis, $a^{n}b^{n}\in\sqrt{I}$ or $b^{n}c^{n}\in\sqrt{I}$ or $a^{n}%
c^{n}\in\sqrt{I}$. Consequently, $ab\in\sqrt{I}$ or $bc\in\sqrt{I}$ or
$ac\in\sqrt{I}$.

(iii):\ Assume that $\sqrt{I}$ is a $\phi$-2-absorbing ideal of $R.$ Choose
$a,b,c\in R$ such that $abc\in I-\phi(I)$. Since $I\subseteq\sqrt{I}$ and
$\phi(\sqrt{I})\subseteq\phi(I),$ we have $abc\in\sqrt{I}-\phi(\sqrt{I}).$
Then as $\sqrt{I}$ is $\phi$-2-absorbing, $ab\in\sqrt{I}$ or $bc\in\sqrt{I}$
or $ac\in\sqrt{I}$ . So $I$ is a $\phi$-quasi primary ideal of $R.$

(iv):\ It is similar to (i).

(v): It is similar to (ii).
\end{proof}

\begin{theorem}
\label{t2}(i) Every $\phi$-quasi primary ideal is a $\phi$-2-absorbing primary ideal.

(ii)\ Every $\phi$-2-absorbing primary ideal is a $\phi$-2-absorbing quasi
primary ideal.

(iii)\ Every $\phi$-quasi primary ideal is a $\phi$-2-absorbing quasi primary ideal.
\end{theorem}

\begin{proof}
(i):\ Let $I$ be a $\phi$-quasi primary ideal and choose $a,b,c\in R\ $such
that $abc=a(bc)\in I-\phi(I).\ $Since $I\ $is a $\phi$-quasi primary ideal, we
conclude either $a\in\sqrt{I}$ or $bc\in\sqrt{I}\ .\ $Then we have either
$ac\in\sqrt{I}\ $or $bc\in\sqrt{I}\ $which completes the proof.

(ii):\ It is clear.

(iii):\ It follows from (i) and (ii).
\end{proof}

By Theorem \ref{t1} and Theorem \ref{t2}, we give the following diagram which
clarifies the place of $\phi$-2-absorbing quasi primary ideal in the lattice
of all ideals $L(R)\ $of $R.$

\begin{figure}[h]
	\includegraphics[width=1.1\textwidth, ]{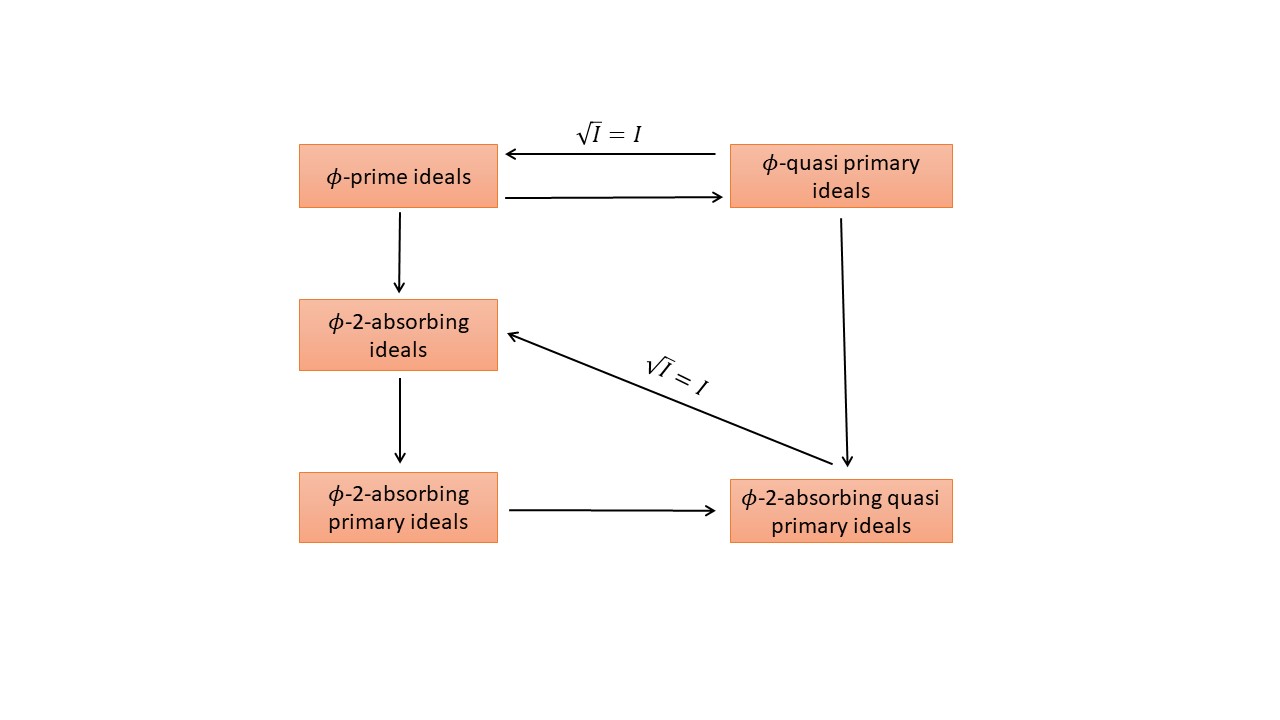}\caption{$\phi
		$-2-absorbing quasi primary ideal vs other classical $\phi$-ideals}%
	\label{fig:figure1}%
\end{figure}

\newpage

\begin{corollary}
If $I$ is a $\phi$-2-absorbing primary ideal of $R$ and $\phi(\sqrt{I}%
)=\sqrt{\phi(I)}$, then $\sqrt{I}$ is a $\phi$-2-absorbing ideal of $R$.
\end{corollary}

\begin{proof}
It follows from Theorem \ref{t1}(ii) and Theorem \ref{t2}(ii).
\end{proof}

\begin{proposition}
\label{pq}Let $I\ $be a proper ideal of $R.$\ Then,

(i)\ $I\ $is a $\phi$-quasi primary ideal of $R\ $if and only if $I/\phi
(I)\ $is a weakly quasi primary ideal of $R/\phi(I).\ $

(ii)\ $I\ $is a $\phi$-2-absorbing quasi primary ideal of $R\ $if and only
$I/\phi(I)\ $is a weakly 2-absorbing quasi primary ideal of $R/\phi(I).$
\end{proposition}

\begin{proof}
(i):\ Suppose that $I\ $is a $\phi$-quasi primary ideal of $R.$ Let
$0_{R/\phi(I)}\neq(a+\phi(I))(b+\phi(I))=ab+\phi(I)\in I/\phi(I)\ $for some
$a,b\in R.\ $Then we have $ab\in I-\phi(I).\ $Since $I\ $is a $\phi$-quasi
primary ideal of $R,\ $we conclude either $a\in\sqrt{I}\ $or $b\in\sqrt{I}%
.\ $This implies that $a+\phi(I)\in\sqrt{I}/\phi(I)=\sqrt{I/\phi(I)}$ or
$b+\phi(I)\in\sqrt{I/\phi(I)}$. Therefore, $I/\phi(I)$is a weakly quasi
primary ideal of $R/\phi(I).\ $Conversely, assume that $I/\phi(I)\ $is a
weakly quasi primary ideal of $R/\phi(I).\ $Now, choose $a,b\in R\ $such that
$ab\in I-\phi(I).\ $This yields that $0_{R/\phi(I)}\neq(a+\phi(I))(b+\phi
(I))=ab+\phi(I)\in I/\phi(I).\ $Since $I/\phi(I)\ $is a weakly quasi primary
ideal of $R/\phi(I),\ $we get either $a+\phi(I)\in\sqrt{I/\phi(I)}=\sqrt
{I}/\phi(I)$ or $b+\phi(I)\in\sqrt{I}/\phi(I)$. Then we have $a\in\sqrt{I}%
\ $or $b\in\sqrt{I}.\ $Hence, $I\ $is a $\phi$-quasi primary ideal of $R.\ $

(ii):\ Similar to (i).
\end{proof}

In the following, we characterize all quasi primary and 2-absorbing quasi
primary ideals in factor rings $R/\phi(I).\ $Since the proof is similar to
that of previous proposition (i), we omit the proof.

\begin{proposition}
\label{pq2}Let $I\ $be a proper ideal of $R.\ $Then,

(i)\ $I\ $is a quasi primary ideal of $R\ $if and only if $I/\phi(I)\ $is a
quasi primary ideal of $R/\phi(I).\ $

(ii)\ $I\ $is a 2-absorbing quasi primary ideal of $R\ $if and only
$I/\phi(I)\ $is a 2-absorbing quasi primary ideal of $R/\phi(I).$
\end{proposition}

Now, we give a method for constructing $\phi$-2-absorbing quasi primary ideal
in commutative rings.

\begin{proposition}
\label{method}Let $P_{1},P_{2}$ be $\phi$-quasi primary ideal of a ring
$R.\ $Then the following statements hold:

(i) If $\phi(P_{1})=\phi(P_{2})=\phi(P_{1}\cap P_{2}),$ then $P_{1}\cap P_{2}$
is a $\phi$-2-absorbing quasi primary ideal of $R.$

(ii) If $\phi(P_{1})=\phi(P_{2})=\phi(P_{1}P_{2}),\ $then $P_{1}P_{2}$ is a
$\phi$-2-absorbing quasi primary ideal of $R$
\end{proposition}

\begin{proof}
(i):\ Let $abc\in P_{1}\cap P_{2}-\phi(P_{1}\cap P_{2})$ for some $a,b,c\in
R.$ Then we have $abc\in P_{1}-\phi(P_{1}).$ As $P_{1}$ is a $\phi$-quasi
primary ideal, we conclude either $a\in\sqrt{P_{1}}$ or $b\in\sqrt{P_{1}}$ or
$c\in\sqrt{P_{1}}$. Similarly, we get either $a\in\sqrt{P_{2}}$ or $b\in
\sqrt{P_{2}}$ or $c\in\sqrt{P_{2}}$. Without loss generality, we may assume
that $a\in\sqrt{P_{1}}$ and $b\in\sqrt{P_{2}}.$ Then $ab\in\sqrt{P_{1}}%
\sqrt{P_{2}}\subseteq\sqrt{P_{1}}\cap\sqrt{P_{2}}=\sqrt{P_{1}\cap P_{2}}.$
Hence, $P_{1}\cap P_{2}$ is a $\phi$-2-absorbing quasi primary ideal of $R.$

(ii):\ Similar to (i).
\end{proof}

\begin{definition}
Let $I\ $be a $\phi$-2-absorbing quasi primary ideal of $R$ and $a,b,c\in
R\ $such that $abc\in\phi(I).\ $If $ab\notin\sqrt{I},\ ac\notin\sqrt{I}\ $and
$bc\notin\sqrt{I},\ $then $(a,b,c)$ is called a strongly-$\phi$-triple zero of
$I.$\ In particular, if $\phi(I)=0,\ $then $(a,b,c)\ $is called a
strongly-triple zero of $I.\ $
\end{definition}

\begin{remark}
If $I\ $is a $\phi$-2-absorbing quasi primary ideal of $R$ that is not a
2-absorbing quasi primary ideal, then I has a strongly-$\phi$-triple zero
$(a,b,c)$ for some $a,b,c\in R.$
\end{remark}

\begin{proposition}
\label{pc}Suppose that $I\ $is a weakly 2-absorbing quasi primary ideal of
$R\ $which is not 2-absorbing quasi primary ideal, then $I^{3}=0.\ $
\end{proposition}

\begin{proof}
\ Let $I\ $be a weakly 2-absorbing quasi primary ideal of $R$ such that
$I^{3}\neq0.\ $Now, we will show that $I\ $is a 2-absorbing quasi primary
ideal of $R.\ $Choose $a,b,c\in R\ $such that $abc\in I.\ $Since $I\ $is
weakly 2-absorbing quasi primary ideal, we may assume that $abc=0.\ $%
Otherwise, we would have $ab\in\sqrt{I}$ or $bc\in\sqrt{I}\ $or $ac\in\sqrt
{I}.\ $If $abI\neq0$,\ then there exists $x\in I\ $such that $abx\neq
0.\ $Since $0\neq abx=ab(c+x)\in I\ $and $I\ $is weakly 2-absorbing quasi
primary ideal, we get either $ab\in\sqrt{I}$ or $a(c+x)\in\sqrt{I}$ or
$b(c+x)\in\sqrt{I}.\ $Thus we have either $ab\in\sqrt{I}$ or $ac\in\sqrt{I}%
\ $or $bc\in\sqrt{I},\ $then we are done. So assume that $abI=0=acI=bcI.$\ On
the other hand, if $aI^{2}\neq0,\ $then there exists $x_{1},x_{2}\in I\ $such
that $ax_{1}x_{2}\neq0.\ $Then we have $0\neq a(b+x_{1})(c+x_{2})=ax_{1}%
x_{2}\in I\ $since $abI=acI=0.\ $As $I\ $is a weakly 2-absorbing quasi primary
ideal, we get either $a(b+x_{1})\in\sqrt{I}\ $or $a(c+x_{2})\in\sqrt{I}\ $or
$(b+x_{1})(c+x_{2})\in\sqrt{I}.\ $Then we have $ab\in\sqrt{I}$ or $bc\in
\sqrt{I}\ $or $ac\in\sqrt{I}.\ $So assume that $aI^{2}=0.\ $Similarly, we may
assume that $bI^{2}=cI^{2}=0.\ $As $I^{3}\neq0,\ $there exist $y,z,w\in
I\ $such that $yzw\neq0.\ $As $abI=0=acI=bcI=aI^{2}=bI^{2}=cI^{2},\ $it is
clear that $0\neq yzw=(a+y)(b+z)(c+w)\in I.\ $This implies that $(a+y)(b+z)\in
\sqrt{I}$ or $(a+y)(c+w)\in\sqrt{I}\ $or $(b+z)(c+w)\in\sqrt{I}$ and so we
have $ab\in\sqrt{I}$ or $bc\in\sqrt{I}\ $or $ac\in\sqrt{I}.\ $Hence, $I\ $is a
2-absorbing quasi primary ideal of $R.$
\end{proof}

\begin{corollary}
\label{cor2}If $I\ $is a weakly 2-absorbing quasi primary ideal of $R\ $which
is not 2-absorbing quasi primary ideal, then $\sqrt{I}=\sqrt{0}.\ $
\end{corollary}

\begin{theorem}
\label{ttc}(i)\ Let $I\ $be a $\phi$-2-absorbing quasi primary ideal of
$R.\ $Then either $I\ $is a 2-absorbing quasi primary ideal or $I^{3}%
\subseteq\phi(I).$

(ii)\ If $I\ $is a $\phi$-2-absorbing quasi primary ideal of $R\ $which is not
2-absorbing quasi primary ideal, then $\sqrt{I}=\sqrt{\phi(I)}.$
\end{theorem}

\begin{proof}
(i)\ Suppose that $I\ $is a $\phi$-2-absorbing quasi primary ideal of $R$ that
is not 2-absorbing quasi primary ideal. Then note that $I/\phi(I)\ $is not a
2-absorbing quasi primary ideal of $R/\phi(I).\ $Also by Proposition \ref{pq},
$I/\phi(I)$ is a weakly 2-absorbing quasi primary ideal of $R/\phi(I).\ $Then
by Proposition \ref{pc}, we get $(I/\phi(I))^{3}=0_{R/\phi(I)}\ $and this
yields $I^{3}\subseteq\phi(I).$

(ii):\ Suppose that $I\ $is a $\phi$-2-absorbing quasi primary ideal of
$R\ $which is not 2-absorbing quasi primary ideal. Then by (i), we have
$I^{3}\subseteq\phi(I)$ and thus $\sqrt{I}\subseteq\sqrt{\phi(I)}.\ $On the
other hand, since $\phi(I)\subseteq I,\ $we have $\sqrt{I}=\sqrt{\phi(I)}.\ $
\end{proof}

\begin{corollary}
\label{equivalent}Suppose that $I\ $is a proper ideal of $R\ $such that
$\phi(I)$ is a quasi primary ideal of $R.\ $Then the following statments are equivalent:

(i)\ $I\ $is a $\phi$-2-absorbing quasi primary ideal of $R.\ $

(ii)\ $I\ $is a 2-absorbing quasi primary ideal of $R.$
\end{corollary}

\begin{proof}
$(i)\Rightarrow(ii):\ $Suppose that $I\ $is a $\phi$-2-absorbing quasi primary
ideal of $R.\ $Now, we will show that $I\ $is a 2-absorbing quasi primary
ideal of $R.$\ Deny. Then Theorem \ref{ttc}\ (ii), we have $\sqrt{I}%
=\sqrt{\phi(I)}.\ $Since $\phi(I)$ is a quasi primary ideal, we have $\sqrt
{I}=\sqrt{\phi(I)}$ is a prime ideal and so $I\ $is quasi primary. Then by
\cite[Proposition 2.6]{qua}, $I\ $is a 2-absorbing quasi primary, a contradiction.

$(ii)\Rightarrow(i):\ $Directly from definition.
\end{proof}

\begin{theorem}
(i)\ If $P$ is a weakly quasi primary ideal of $R$ that is not quasi primary,
then $P^{2}=0.$

(ii) If\ $P$ is a $\phi$-quasi primary ideal of $R$ that is not quasi primary,
then $P^{2}\subseteq\phi(P).$

(iii)\ If $P$ is a $\phi$-quasi primary ideal of $R$ where $\phi\leq\phi_{3},$
then $P$ is $n$-almost quasi primary for all $n\geq2,$ so $\omega$-quasi primary.
\end{theorem}

\begin{proof}
(i):\ Similar to Proposition \ref{pc}.

(ii):\ Similar to Theorem \ref{ttc} (i).

(ii): Assume that $P$ is a $\phi$-quasi primary ideal of $R$ and $\phi\leq
\phi_{3}.$ If $P$ is quasi primary, then $P$ is $\phi$-quasi primary for each
$\phi$. If $P$ is not quasi primary, by (i), $P^{2}\subseteq\phi(P).$ Also as
$\phi\leq\phi_{3},$ we get $P^{2}\subseteq\phi(P)\subseteq P^{3},$ so
$\phi(P)=P^{n}$ for each $n\geq2.$ Consequently, since $P$ is $\phi$-quasi
primary, $P$ is $n$-almost quasi primary for all $n\geq2,$ so $\omega$-quasi
primary by Proposition \ref{pp}\ (iii).
\end{proof}

Now, we give the Nakayama's Lemma for weakly (2-absorbing) quasi primary
ideals as follows.

\begin{theorem}
\label{nakayama}(Nakayama's Lemma) Let $P\ $be a weakly 2-absorbing quasi
primary (weakly quasi primary) ideal of $R\ $that is not 2-absorbing quasi
primary (quasi primary) and $M\ $be an $R$-module. Then the following
statements hold:

(i)\ $P\subseteq Jac(R)$,\ where $Jac(R)\ $is the jacobson radical of $R.$

(ii)\ If $PM=M,\ $then $M=0.$

(iii)\ If $N\ $is a submodule of $M\ $such that $PM+N=M,\ $then $N=M.\ $
\end{theorem}

\begin{proof}
$(i):\ $Suppose that $P\ $is a weakly 2-absorbing quasi primary ideal of $R$
that is not 2-absorbing quasi primary$.\ $Then by Theorem \ref{pc}, $P^{3}%
=0$.\ Let $x\in P.\ $Now, we will show that $1-rx$ is a unit of $R\ $for each
$r\in R.\ $Note that $rx\in P\ $and so $r^{3}x^{3}=0.\ $This implies that
$1=1-r^{3}x^{3}=(1-rx)(1+rx+r^{2}x^{2}).\ $Thus $x\in Jac(R)\ $and so
$P\subseteq Jac(R).\ $

$(ii):\ $Suppose that $PM=M.\ $Then by Theorem \ref{pc}, $P^{3}=0$ and so
$M=PM=P^{3}M=0.\ $

$(iii):\ $Follows from (ii).\ 
\end{proof}

\begin{theorem}
\label{quo} Let $S$ be a multiplicatively closed subset of $R$ and $\phi
_{q}:L(S^{-1}R)\rightarrow L(S^{-1}R)\cup\{\emptyset\},\ $defined by $\phi
_{q}(S^{-1}I)=S^{-1}(\phi(I))$ for each ideal $I\ $of $R,\ $be a function.
Then the following statements hold:

(i) If $P$ is a $\phi$-2-absorbing quasi primary ideal of $R$ with $S\cap
P=\emptyset,$ then $S^{-1}P$ is a $\phi_{q}$-2-absorbing quasi primary ideal
of $S^{-1}R.$

(ii)\ Let $P$ be an ideal of $R$ such that $Z_{\phi(P)}(R)\cap S=\emptyset$
and $Z_{P}(R)\cap S=\emptyset.$ If $S^{-1}P$ is a $\phi_{q}$-2-absorbing quasi
primary ideal of $S^{-1}R,$ then $P$ is a $\phi$-2-absorbing quasi primary
ideal of $R.$
\end{theorem}

\begin{proof}
(i):\ Let $\frac{a}{s}\frac{b}{t}\frac{c}{u}\in S^{-1}P-\phi_{q}(S^{-1}P)$ for
any $a,b,c\in R$ and $s,t,u\in S.$ As $\phi_{q}(S^{-1}P)=S^{-1}(\phi(P)),$ we
get $t^{\ast}abc=(t^{\ast}a)bc\in P-\phi(P)$ for some $t^{\ast}\in S.$ Since
$P$ is a $\phi$-2-absorbing quasi primary ideal of $R,$ we get $t^{\ast}%
ab\in\sqrt{P}$ or $t^{\ast}ac\in\sqrt{P}$ or $bc\in\sqrt{P}.$ This implies
that $\frac{ab}{st}=\frac{t^{\ast}ab}{t^{\ast}st}\in S^{-1}\sqrt{P}%
=\sqrt{S^{-1}P}$ or $\frac{ac}{su}=\frac{t^{\ast}ac}{t^{\ast}su}\in
\sqrt{S^{-1}P}$ or $\frac{bc}{tu}\in\sqrt{S^{-1}P}.$ Hence $S^{-1}P$ is a
$\phi_{q}$-2-absorbing quasi primary ideal of $S^{-1}R.$

(ii):\ Let $abc\in P-\phi(P)$ for some $a,b,c\in R.$ Then $\frac{a}{1}\frac
{b}{1}\frac{c}{1}\in S^{-1}P.$ Since $Z_{\phi(P)}(R)\cap S=\emptyset,$ it is
clear that $\frac{a}{1}\frac{b}{1}\frac{c}{1}\notin S^{-1}(\phi(P))=\phi
_{q}(S^{-1}P).$ As $S^{-1}P$ is a $\phi_{q}$-2-absorbing quasi primary ideal
of $S^{-1}R,$ we get either $\frac{a}{1}\frac{b}{1}\in\sqrt{S^{-1}P}%
=S^{-1}\sqrt{P}$ or $\frac{a}{1}\frac{c}{1}\in S^{-1}\sqrt{P}$ or $\frac{b}%
{1}\frac{c}{1}\in S^{-1}\sqrt{P}.$ Without loss generality, we may assume that
$\frac{b}{1}\frac{c}{1}\in S^{-1}\sqrt{P}.$ Then we get $tbc\in\sqrt{P}$ and
so $t^{n}b^{n}c^{n}\in P$ for some $n\in%
\mathbb{N}
.$ If $b^{n}c^{n}\notin P,$ then we have $t^{n}\in Z_{P}(R)\cap S,$ a
contradiction. So we have $b^{n}c^{n}\in P$ and thus $bc\in\sqrt{P}.$ Thus $P$
is a $\phi$-2-absorbing quasi primary ideal of $R.$
\end{proof}

Let $M\ $be an $R$-module. The trivial ring extension (or idealization)
$R\propto M=R\oplus M$ \ of $M\ $is a commutative ring with componentwise
addition and multiplication $(a,m)(b,m^{\prime})=(ab,am^{\prime}+bm)$ for each
$a,b\in R;\ m,m^{\prime}\in M\ $\cite{Huc}.\ Let $I\ $be an ideal of $R\ $and
$N\ $is a submodule of $M.\ $Then $I\propto N=I\oplus N\ $is an ideal of
$R\propto M$ if and only if $IM\subseteq N$ \cite{AnWi}$.\ $In that case,
$I\propto N$ is called a homogeneous ideal of $R\propto M.\ $For any ideal
$I\propto N\ $of $R\propto M,\ $the radical of $I\propto N$ is characterized
as follows:
\[
\sqrt{I\propto N}=\sqrt{I}\propto M
\]
\cite[Theorem 3.2]{AnWi}.

Now, we characterize certain weakly 2-absorbing quasi primary ideals in
trivial ring extensions.

\begin{theorem}
\label{idealization}Let $M\ $be an $R$-module and $I\ $a proper ideal of
$R.\ $Then the following statements are equivalent:

(i)\ $I\propto M\ $is a weakly 2-absorbing quasi primary ideal of $R\propto
M.$

(ii)\ $I\ $is a weakly 2-absorbing quasi primary ideal of $R\ $and for any
strongly triple zero $(a,b,c)$ of $I,\ $we have $abM=0=acM=bcM.$
\end{theorem}

\begin{proof}
$(i)\Rightarrow(ii):\ $Suppose that $I\propto M\ $is a weakly 2-absorbing
quasi primary ideal of $R\propto M.\ $Now, we will show that $I\ $is a weakly
2-absorbing quasi primary ideal of $R.\ $Let $0\neq abc\in I.$\ Then note that
$(0,0_{M})\neq(a,0_{M})(b,0_{M})(c,0_{M})=(abc,0_{M})\in I\propto M.\ $As
\ $I\propto M\ $is a weakly 2-absorbing quasi primary ideal, we conclude
either $(a,0_{M})(b,0_{M})=(ab,0_{M})\in\sqrt{I\propto M}=\sqrt{I}\propto M$
or $(ac,0_{M})\in\sqrt{I}\propto M$ or $(bc,0_{M})\in\sqrt{I}\propto M.\ $This
implies that $ab\in\sqrt{I}$ or $ac\in\sqrt{I}\ $or $bc\in\sqrt{I}%
.\ $Therefore, $I\ $is a weakly 2-absorbing quasi primary ideal of $R.\ $Let
$(x,y,z)$ be a strongly triple zero of $I.\ $Then $xyz=0\ $and also
$xy,xz,yz\notin\sqrt{I}.\ $Assume that $xyM\neq0.\ $Then there exists $m\in
M\ $such that $xym\neq0.\ $Then note that $(0,0_{M})\neq(x,0_{M}%
)(y,0_{M})(z,m)=(0,xym)\in I\propto M.\ $Since $I\propto M\ $is a weakly
2-absorbing quasi primary ideal, we conclude either $(x,0_{M})(y,0_{M}%
)=(xy,0_{M})\in\sqrt{I\propto M}=\sqrt{I}\propto M$ or $(xz,xm)\in\sqrt
{I}\propto M$ or $(yz,ym)\in\sqrt{I}\propto M,\ $a contradiction. Thus
$xM=yM=zM=0.\ $

$(ii)\Rightarrow(i):\ $Suppose that $(0,0_{M})\neq(a,m)(b,m^{\prime
})(c,m^{\prime\prime})=(abc,abm^{\prime\prime}+acm^{\prime}+bcm)\in I\propto
M$ for some $a,b,c\in R;\ m,m^{\prime},m^{\prime\prime}\in M.\ $Then $abc\in
I.\ $If $abc\neq0,\ $then either $ab\in\sqrt{I}$ or $ac\in\sqrt{I}\ $or
$bc\in\sqrt{I}.\ $This implies that $(a,m)(b,m^{\prime})\in\sqrt{I\propto
M}=\sqrt{I}\propto M$ or $(a,m)(c,m^{\prime\prime})\in\sqrt{I}\propto M$ or
$(b,m^{\prime})(c,m^{\prime\prime})\in\sqrt{I}\propto M$.\ Now assume that
$abc=0.\ $If $(a,b,c)$ is a strongly triple zero of $I,\ $then by assumption
$abM=0=acM=bcM\ $and so $(0,0_{M})=(abc,abm^{\prime\prime}+acm^{\prime
}+bcm)=(a,m)(b,m^{\prime})(c,m^{\prime\prime})$ which is a contradiction. So
that $(a,b,c)$ is not strongly triple zero and this yields $ab\in\sqrt{I}$ or
$ac\in\sqrt{I}\ $or $bc\in\sqrt{I}.\ $Therefore, we have $(a,m)(b,m^{\prime
})\in\sqrt{I\propto M}$ or $(a,m)(c,m^{\prime\prime})\in\sqrt{I\propto M}$ or
$(b,m^{\prime})(c,m^{\prime\prime})\in\sqrt{I\propto M}.\ $Hence, $I\propto
M\ $is a weakly 2-absorbing quasi primary ideal of $R\propto M.$
\end{proof}

Let $R_{1},R_{2},\ldots,R_{n}$ be commutative rings and $R=R_{1}\times
R_{2}\times\cdots\times R_{n}$ be the direct product of those rings. It is
well known that every ideal of $R$ has the form $I=I_{1}\times I_{2}%
\times\cdots\times I_{n},\ $where $I_{k}\ $is an ideal of $R_{k}$ for each
$1\leq k\leq n.\ $Suppose that $\psi_{i}:L(R_{i})\rightarrow L(R_{i}%
)\cup\{\emptyset\}$ is a function for each $1\leq i\leq n\ $and put
$\phi:=\psi_{1}\times\psi_{2}\times\cdots\times\psi_{n},\ $that is,
$\phi(I)=\psi_{1}(I_{1})\times\psi_{2}(I_{2})\times\cdots\times\psi_{n}%
(I_{n}).$\ Then note that $\phi:L(R)\rightarrow L(R)\cup\{\emptyset\}$ becomes
a function.

Recall that a commutative ring $R\ $is said to be a \textit{von Neumann
regular ring }if for each $a\in R,\ $there exists $x\in R\ $such that
$a=a^{2}x$ \cite{von}$.\ $In this case, the principal ideal $(a)$ is a
generated by an idempotent element $e\in R.\ $The notion of von Neumann
regular ring has an important place in commutative algebra. So far, there have
been many generalizations of this concept. See, for example, \cite{JaTe},
\cite{AnChu} and \cite{Al}. Now, we characterize von Neumann regular rings in
terms of $\phi$-2-absorbing quasi primary ideals.

\begin{theorem}
\label{VonNue}Let $R_{1},R_{2},\ldots,R_{m}\ $be commutative rings and
$R=R_{1}\times R_{2}\times\cdots\times R_{m},\ $where $3\leq m<\infty.$
Suppose that $n\geq2.\ $Then the following statements are equivalent.

(i)\ Every ideal of $R\ $is a $\phi_{n}$-2-absorbing quasi primary ideal.

(ii)\ $R_{1},R_{2},\ldots,R_{m}\ $are von Neumann regular rings.
\end{theorem}

\begin{proof}
$(i)\Rightarrow(ii):\ $Suppose that every ideal of $R\ $is a $\phi_{n}%
$-2-absorbing quasi primary ideal. Now, we will show that \ $R_{1}%
,R_{2},\ldots,R_{m}\ $are von Neumann regular rings. Suppose not. Without loss
of generality, we may assume that $R_{1}$ is not a von Neumann regular ring.
Then there exists an ideal $I_{1}\ $of $R_{1}$ such that $I_{1}^{n}%
\varsubsetneq I_{1}.\ $Then we can find an element $a\in I_{1}-I_{1}^{n}%
.\ $Now, put $J=I_{1}\times0\times0\times R_{4}\times R_{5}\times\cdots\times
R_{m}\ $and also $x_{1}=(a,1,1,1,\ldots,1),\ x_{2}=(1,0,1,1,\ldots
,1),\ x_{3}=(1,1,0,1,\ldots,1).\ $Then note that $x_{1}x_{2}x_{3}\in
J-\phi_{n}(J).\ $As $x_{1}x_{2},x_{1}x_{3}$ and $x_{2}x_{3}\notin\sqrt
{J},\ J\ $is not a $\phi_{n}$-2-absorbing quasi primary ideal of $R$ which is
a contradiction. Thus $R_{1},R_{2},\ldots,R_{m}\ $are von Neumann regular rings.

$(ii)\Rightarrow(i):\ $Suppose that $R_{1},R_{2},\ldots,R_{m}\ $are von
Neumann regular rings.\ Then note that $I_{i}^{n}=I_{i}\ $for any ideal
$I_{i}\ $of $R_{i}.\ $Take any ideal $J\ $of $R.\ $Then $J=J_{1}\times
J_{2}\times\cdots\times J_{m}\ $for some ideal $J_{k}\ $of $R_{k},\ $where
$1\leq k\leq m.\ $Then $J^{n}=J_{1}^{n}\times J_{2}^{n}\times\cdots\times
J_{m}^{n}=\phi_{n}(J)=J_{1}\times J_{2}\times\cdots\times J_{m}=J.\ $This
implies that $J-\phi_{n}(J)=\emptyset\ $and so $J\ $is trivially $\phi_{n}%
$-2-absorbing quasi primary ideal of $R.\ $
\end{proof}

\begin{corollary}
Let $R$ be a ring and $n\geq2.\ $Then the following statements are equivalent:

(i) Every ideal of\ $R^{3}\ $is a $\phi_{n}$-2-absorbing quasi primary ideal.

(ii)\ $R\ $is a von Neumann regular ring.
\end{corollary}

\section{$\phi$\textbf{-(2-absorbing) quasi primary ideals in direct product
of rings}}

In this section, we investigate $\phi$-2-absorbing quasi primary ideal in
direct product of finitely many commutative rings.

\begin{theorem}
\label{tsi}Let $R=R_{1}\times R_{2}$, where $R_{1}$ and $R_{2}$ are rings and
$\psi_{i}:L(R_{i})\rightarrow L(R_{i})\cup\{\emptyset\}$ is a function for
each $i=1,2$. Let $\phi:=\psi_{1}\times\psi_{2}$ and $J\ $an ideal of
$R.$\ Then $J$ is a $\phi$-quasi primary ideal of $R$ if and only if $J$ has
exactly one of the following three types:

(i) $J=I_{1}\times I_{2}$ such that $\psi_{i}(I_{i})=I_{i}$ for $i=1,2$.

(ii)$\ I_{1}\times R_{2}$ for some $\psi_{1}$-quasi primary ideal $I_{1}$ of
$R_{1}$ which must be quasi primary if $\psi_{2}(R_{2})\neq R_{2}$.

(iii)$\ R_{1}\times I_{2}$ for some $\psi_{2}$-quasi primary ideal $I_{2}$ of
$R_{2}$ which must be quasi primary if $\psi_{1}(R_{1})\neq R_{1.}$
\end{theorem}

\begin{proof}
$\Rightarrow:$ Suppose that $J$ is a $\phi$-quasi primary ideal of $R.$ Then
$J=I_{1}\times I_{2}$ for some ideal $I_{1}$ of $R_{1}$ and some ideal $I_{2}$
of $R_{2}$. Let $xy\in I_{1}-\psi_{1}(I_{1}).\ $Then we have
$(x,0)(y,0)=(xy,0)\in J-\phi(J).\ $As\ $J$ is a $\phi$-quasi primary ideal, we
conclude either $(x,0)\in\sqrt{J}\ $or $(y,0)\in\sqrt{J}.\ $Since $\sqrt
{J}=\sqrt{I_{1}}\times\sqrt{I_{2}},\ $we get $x\in\sqrt{I_{1}}$ or $y\in
\sqrt{I_{1}}.\ $Hence, $I_{1}$\ is a $\psi_{1}$-quasi primary ideal.
Similarly, $I_{2}\ $is a $\psi_{2}$-quasi primary ideal. We may assume that
$J\neq\phi(J).\ $Then we have either $I_{1}\neq\psi_{1}(I_{1})$ or $I_{2}%
\neq\psi_{2}(I_{2}).\ $Without loss of generality, we may assume that
$I_{1}\neq\psi_{1}(I_{1}).\ $So there exists $a\in I_{1}-\psi_{1}(I_{1}%
).\ $Take $b\in I_{2}.\ $Then we have $(a,1)(1,b)\in J-\phi(J).\ $This implies
either $(a,1)\in\sqrt{J}$ or $(1,b)\in\sqrt{J}.$\ Then we get $1\in\sqrt
{I_{1}}\ $or $1\in\sqrt{I_{2}},\ $that is, $I_{1}=R_{1}\ $or $I_{2}=R_{2}%
.\ $Now, assume that $I_{2}=R_{2}.\ $Now, we will show that $I_{1}\ $is a
quasi primary ideal provided that $\psi_{2}(R_{2})\neq R_{2}.\ $So suppose
$\psi_{2}(R_{2})\neq R_{2}.\ $Let $xy\in I_{1}\ $for some $x,y\in R_{1}%
.\ $Then we have $(x,1)(y,1)=(xy,1)\in I_{1}\times R_{2}-\phi(I_{1}\times
R_{2}).\ $As $J$ is a $\phi$-quasi primary ideal, we get either $(x,1)\in
\sqrt{J}\ $or $(y,1)\in\sqrt{J}.\ $Hence, $x\in\sqrt{I_{1}}$ or $y\in
\sqrt{I_{1}}.\ $Therefore, $I_{1}\ $is a quasi primary ideal.

$\Leftarrow:\ $Suppose that $J=I_{1}\times I_{2}$ such that $\psi_{i}%
(I_{i})=I_{i}\ $for $i=1,2$. Since $\phi(I_{1}\times I_{2})=\psi_{1}%
(I_{1})\times\psi_{2}(I_{2})=I_{1}\times I_{2},$ we get $I_{1}\times
I_{2}-\phi(I_{1}\times I_{2})=\emptyset\ $and so $J\ $is trivially $\phi
$-quasi primary ideal. Let $J=I_{1}\times R_{2}$ for some $\psi_{1}$-quasi
primary ideal $I_{1}$ of $R_{1}$\ which must be quasi primary if $\psi
_{2}(R_{2})\neq R_{2}$.\ First, assume that $\psi_{2}(R_{2})=R_{2}.\ $Then
note that $\phi(J)=\psi_{1}(I_{1})\times R_{2}.\ $Let $(x_{1},x_{2}%
)(y_{1},y_{2})=(x_{1}y_{1},x_{2}y_{2})\in J-\phi(J)\ $for some $x_{i},y_{i}\in
R_{i}.\ $Then we have $x_{1}y_{1}\in I_{1}-\psi_{1}(I_{1}).\ $This yields that
$x_{1}\in\sqrt{I_{1}}$ or $y_{1}\in\sqrt{I_{1}}\ $since $I_{1}$\ is a
$\psi_{1}$-quasi primary ideal.\ Then we get either $(x_{1},x_{2})\in
\sqrt{I_{1}\times R_{2}}=\sqrt{I_{1}}\times R_{2}$ or $(y_{1},y_{2})\in
\sqrt{I_{1}\times R_{2}}.$\ Hence, $J$ is a $\phi$-quasi primary ideal of
$R.$\ Now, assume that $\psi_{2}(R_{2})\neq R_{2}$\ and $I_{1}\ $is a quasi
primary ideal. Then $I_{1}\times R_{2}\ $is a quasi primary ideal of $R\ $by
\cite[Lemma 2.2]{qua}. Hence, $J=I_{1}\times R_{2}\ $is a $\phi$-quasi primary
ideal of $R.\ $In other case, one can see that $J$\ is a $\phi$-quasi primary
ideal of $R.$
\end{proof}

\begin{theorem}
Let $R=R_{1}\times R_{2}\times\cdots\times R_{n}$, where $R_{1},R_{2}%
,\ldots,R_{n}$ are rings and $\psi_{i}:L(R_{i})\rightarrow L(R_{i}%
)\cup\{\emptyset\}$ be a function for each $i=1,2,\ldots,n$. Let $\phi
:=\psi_{1}\times\psi_{2}\times\cdots\times\psi_{n}$ and $J\ $be an ideal of
$R.$\ Then $J$ is a $\phi$-quasi primary ideal of $R$ if and only if $J$ has
exactly one of the following two types:

(i)\ $J=I_{1}\times I_{2}\times\cdots\times I_{n}$ such that $\psi_{i}%
(I_{i})=I_{i}$ for $i=1,2,\ldots,n$.

(ii)\ $J=R_{1}\times R_{2}\times\cdots\times R_{t-1}\times I_{t}\times
R_{t+1}\times\cdots\times R_{n}\ $for some $\psi_{t}$-quasi primary ideal
$I_{t}$ of $R_{t}\ $which must be quasi primary if $\psi_{j}(R_{j})\neq
R_{j}\ $for some $j\neq t.$
\end{theorem}

\begin{proof}
We use induction on $n\ $to prove the claim. If $n=1,\ $the claim is clear. If
$n=2,\ $the claim follows from previous theorem. Assume that the claim is true
for all $n<k\ $and put $n=k.\ $Put $R^{\prime}=R_{1}\times R_{2}\times
\cdots\times R_{k-1},\ J^{\prime}=I_{1}\times I_{2}\times\cdots\times
I_{k-1}\ $and $\phi^{\prime}=\psi_{1}\times\psi_{2}\times\cdots\times
\psi_{k-1}.\ $Then note that $R=R^{\prime}\times R_{k},\ J=J^{\prime}\times
J_{k}\ $and $\phi=\phi^{\prime}\times\psi_{k}.\ $Then by previous theorem,
$J\ $is a $\phi$-quasi primary ideal of $R$ if and only if one of the
following conditions hold: (i) $J=J^{\prime}\times I_{k}$ such that
$\phi^{\prime}(J^{\prime})=J^{\prime}$ and $\psi_{k}(I_{k})=I_{k}\ $(ii)
$J=J^{\prime}\times R_{k}$ for some $\phi^{\prime}$-quasi primary ideal
$J^{\prime}$ of $R^{\prime}$ which must be quasi primary if $\psi_{k}%
(R_{k})\neq R_{k}$\ (iii)\ $J=R^{\prime}\times I_{k}$ for some $\psi_{k}%
$-quasi primary ideal $I_{k}$ of $R_{k}$ which must be quasi primary if
$\phi^{\prime}(R^{\prime})\neq R^{\prime}.\ $The rest follows from induction
hypothesis and \cite[Theorem 2.3]{qua}.{}
\end{proof}

\begin{theorem}
\label{car}Let$\ R_{1}$\ and $R_{2}$ be commutative rings with identity and
let $R=R_{1}\times R_{2}$. Suppose that $\psi_{i}:L(R_{i})\rightarrow
L(R_{i})\cup\{\emptyset\}$ $(i=1,2)$ are functions such that $\psi_{2}%
(R_{2})\neq R_{2}$\ and $\phi=\psi_{1}\times\psi_{2}.$ Then the following
assertions are equivalent:

(i)$\ I_{1}\times R_{2}$ is a $\phi$-2-absorbing quasi primary ideal of $R.$

(ii)$\ I_{1}\times R_{2}$ is a 2-absorbing quasi primary ideal of $R.$

(iii)$\ I_{1}$ is a 2-absorbing quasi primary ideal of $R_{1}.$
\end{theorem}

\begin{proof}
Assume that $\psi_{1}(I_{1})=\emptyset$ or $\psi_{2}(R_{2})=\emptyset.$ Then
clearly $\phi(I_{1}\times R_{2})=\emptyset$ so that $(i)\Leftrightarrow
(ii)\Leftrightarrow(iii)\ $follows from \cite[Theorem 2.23]{qua}. Hence
suppose that $\psi_{1}(I_{1})\neq\emptyset$ and $\psi_{2}(R_{2})\neq
\emptyset,$ so $\phi(I_{1}\times R_{2})\neq\emptyset.$

$(i)\Rightarrow(ii):$ Suppose that $I_{1}\times R_{2}$ is a $\phi$-2-absorbing
quasi primary ideal of $R.$ Similar argument in the proof of Theorem \ref{tsi}
shows that $I_{1}$ is a $\psi_{1}$-2-absorbing quasi primary ideal of $R_{1}.$
If $I_{1}$ is 2-absorbing quasi primary, then $I_{1}\times R_{2}$ is a
2-absorbing quasi primary ideal of $R,$ by \cite[Theorem 2.23]{qua}.\ If
$I_{1}$ is not 2-absorbing quasi primary, then $I_{1}$ has a strongly
$\psi_{1}$-triple zero $(x,y,z)$ for some $x,y,z\in R_{1}$ by Remark 1. Then
$(x,1)(y,1)(z,1)=(xyz,1)\in I_{1}\times R_{2}-\psi_{1}(I_{1})\times\psi
_{2}(R_{2})$ since $\psi_{2}(R_{2})\neq R_{2}.$ This implies that $xy\in
\sqrt{I_{1}}\ $or $yz\in\sqrt{I_{1}}$ or $xz\in\sqrt{I_{1}},$ a contradiction.
Thus $I_{1}$ is 2-absorbing quasi primary. Consequently, $I_{1}\times R_{2}$
is a 2-absorbing quasi primary ideal of $R.$

$(ii)\Rightarrow(iii)$ and $(iii)\Rightarrow(i):$ Follows from \cite[Theorem
2.23]{qua}.
\end{proof}

\begin{theorem}
\label{car2}Let $R_{1}$ and $R_{2}$ be commutative rings with identity and let
$R=R_{1}\times R_{2}$. Suppose that $\psi_{i}:S(R_{i})\rightarrow S(R_{i}%
)\cup\{\emptyset\}$ $(i=1,2)$ are functions and $\phi=\psi_{1}\times\psi_{2}.$
The following statements are equivalent:

(i)$\ I_{1}\times R_{2}$ is a $\phi$-2-absorbing quasi primary ideal of $R$
that is not a 2-absorbing quasi primary ideal of $R.$

(ii)$\ \phi(I_{1}\times R_{2})\neq\emptyset,\psi_{2}(R_{2})=R_{2}$ and $I_{1}$
is a $\psi_{1}$-2-absorbing quasi primary ideal of $R_{1}$ that is not a
2-absorbing quasi primary ideal of $R_{1}.$
\end{theorem}

\begin{proof}
$(i)\Rightarrow(ii):$ Let $I_{1}\times R_{2}$ be $\phi$-2-absorbing quasi
primary that is not 2-absorbing quasi primary. By Theorem \ref{car}, since
$I_{1}\times R_{2}$ is not a 2-absorbing quasi primary ideal of $R,$ one can
see that $\phi(I_{1}\times R_{2})\neq\emptyset$ and $\psi_{2}(R_{2})=R_{2}.$
As $I_{1}\times R_{2}$ is a $\phi$-2-absorbing quasi primary ideal of $R,$ it
is clear that $I_{1}$ is a $\psi_{1}$-2-absorbing quasi primary ideal of
$R_{1}.$ Also, since $I_{1}\times R_{2}$ is not a 2-absorbing quasi primary
ideal of $R,$ $I_{1}$ is not a 2-absorbing quasi primary ideal of $R_{1}$ by
\cite[Theorem 2.3]{qua}.

$(ii)\Rightarrow(i):$ Since $\phi(I_{1}\times R_{2})\neq\emptyset$ and
$\psi_{2}(R_{2})=R_{2},$ we get $R/\phi(I_{1}\times R_{2})\cong R_{1}/\psi
_{1}(R_{1})$ and $I_{1}\times R_{2}/\phi(I_{1}\times R_{2})\cong I_{1}%
/\psi_{1}(I_{1}).\ $By Proposition \ref{pq}(ii), since\ $I_{1}\ $is a
$\psi_{1}$-2-absorbing quasi primary ideal of $R_{1},$ $I_{1}/\psi_{1}%
(I_{1})\ $is a weakly 2-absorbing quasi primary ideal of $R_{1}/\psi_{1}%
(R_{1}).$ Also, as $I_{1}$ is not a 2-absorbing quasi primary ideal of
$R_{1},$ then $I_{1}/\psi_{1}(I_{1})\ $is not a 2-absorbing quasi primary
ideal of $R_{1}/\psi_{1}(R_{1}),$ by Proposition \ref{pq2}(ii). Thus,
$I_{1}\times R_{2}/\phi(I_{1}\times R_{2})$ is a weakly 2-absorbing quasi
primary ideal of $R/\phi(I_{1}\times R_{2})$ that is not a 2-absorbing quasi
primary. Consequently, again by Proposition \ref{pq}(ii) and Proposition
\ref{pq2}(ii), we obtain that $I_{1}\times R_{2}$ is a $\phi$-2-absorbing
quasi primary ideal of $R$ that is not a 2-absorbing quasi primary ideal of
$R.$
\end{proof}

The following Theorem is a consequence of Theorem \ref{car}.

\begin{theorem}
\label{weak1}Let $R_{1}$ and $R_{2}$ be commutative rings with identity and
let $R=R_{1}\times R_{2}.$ Then the following assertions are equivalent:

(i) $I_{1}\times R_{2}$ is a weakly 2-absorbing quasi primary ideal of $R.$

(ii) $I_{1}\times R_{2}$ is a 2-absorbing quasi primary ideal of $R.$

(iii)$\ I_{1}$ is a 2-absorbing quasi primary ideal of $R_{1}.$
\end{theorem}

\begin{theorem}
\label{weak2}Let $R_{1}$ and $R_{2}$ be commutative rings with identity and
$R=R_{1}\times R_{2}.$ Let $I_{1}\times I_{2}$ be a proper ideal of
$R,\ $where $I_{1},I_{2}\ $are nonzero ideals of $R_{1}\ $and $R_{2}%
,\ $respectively.\ Then the following assertions are equivalent:

(i)$\ I_{1}\times I_{2}$ is a weakly 2-absorbing quasi primary ideal of $R.$

(ii)$\ I_{1}\times I_{2}$ is a 2-absorbing quasi primary ideal of $R.$

(ii)$\ I_{1}=R_{1}$ and $I_{2}$ is a 2-absorbing quasi primary ideal of
$R_{2}$ or $I_{2}=R_{2}$ and $I_{1}$ is a 2-absorbing quasi primary ideal of
$R_{1}$ or $I_{1},I_{2}$ are quasi primary of $R_{1},R_{2},$ respectively.
\end{theorem}

\begin{proof}
$(i)\Rightarrow(iii):$ Suppose that $I_{1}\times I_{2}$ is a weakly
2-absorbing quasi primary ideal of $R.$ If $I_{1}=R_{1},$ by Theorem
\ref{weak1}, $I_{2}$ is a 2-absorbing quasi primary ideal of $R_{2}.$
Similarly, if $I_{2}=R_{2},$ $I_{1}$ is a 2-absorbing quasi primary ideal of
$R_{1}.$ Thus we may assume that $I_{1}\neq R_{1}$ and $I_{2}\neq R_{2}.$ Let
us show $I_{2}$ is a quasi primary ideal of $R_{2}$. Take $x,y\in R_{2}$ such
that $xy\in I_{2}.$ Choose $0\neq a\in I_{1}.$ Then $0\neq
(a,1)(1,x)(1,y)=(a,xy)\in I_{1}\times I_{2}.$ By our hypothesis,
$(a,x)\in\sqrt{I_{1}\times I_{2}}=\sqrt{I_{1}}\times\sqrt{I_{2}}$ or
$(1,xy)\in\sqrt{I_{1}}\times\sqrt{I_{2}}$ or $(a,y)\in\sqrt{I_{1}}\times
\sqrt{I_{2}}.$ If $(1,xy)\in\sqrt{I_{1}}\times\sqrt{I_{2}}$, a contradiction
(as $I_{1}\neq R_{1}$). Thus we obtain that $(a,x)\in\sqrt{I_{1}}\times
\sqrt{I_{2}}$ or $(a,y)\in\sqrt{I_{1}}\times\sqrt{I_{2}}.$ This implies that
$x\in\sqrt{I_{2}}$\ or $y\in\sqrt{I_{2}}.$ Similarly, we can show that $I_{1}$
is a quasi primary ideal of $R_{1}.$

$(ii)\Leftrightarrow(iii):$ By \cite[Theorem 2.23]{qua}.

$(ii)\Rightarrow(i):$ It is clear.
\end{proof}

\begin{theorem}
\label{weak3}Let $R_{1}$ and $R_{2}$ be commutative rings with identity and
$R=R_{1}\times R_{2}.$ Then a non-zero ideal $I_{1}\times I_{2}$ of $R$ is
weakly 2-absorbing quasi primary that is not 2-absorbing quasi primary if and
only if one of the following assertions holds:

(i)$\ I_{1}\neq R_{1}$ is a nonzero weakly quasi primary ideal of $R_{1}$ that
is not quasi primary and $I_{2}=0$ is a quasi primary ideal of $R_{2}.$

(ii)$\ I_{2}\neq R_{2}$ is a nonzero weakly quasi primary ideal of $R_{2}$
that is not quasi primary and $I_{1}=0$ is a quasi primary ideal of $R_{1}.$
\end{theorem}

\begin{proof}
Assume that $I_{1}\times I_{2}$ is a weakly 2-absorbing quasi primary ideal of
$R$ that is not 2-absorbing quasi primary. Suppose that $I_{1}\neq0$ and
$I_{2}\neq0.$ By Theorem \ref{weak2}, $I_{1}\times I_{2}$ is 2-absorbing quasi
primary, a contradiction. Thus $I_{1}=0$ or $I_{2}=0.$ Without loss of
generality, suppose that $I_{2}=0.$ Let us prove that $I_{2}=0$ is a quasi
primary ideal of $R_{2}.$ Choose $x,y\in R_{2}$ such that $xy\in I_{2}.$ Take
$0\neq a\in I_{1}.$ Then $0\neq(a,1)(1,x)(1,y)=(a,xy)\in I_{1}\times I_{2}.$
By our hypothesis, $(a,x)\in\sqrt{I_{1}\times I_{2}}=\sqrt{I_{1}}\times
\sqrt{I_{2}}$ or $(1,xy)\in\sqrt{I_{1}}\times\sqrt{I_{2}}$ or $(a,y)\in
\sqrt{I_{1}}\times\sqrt{I_{2}}.$ Here $(1,xy)\notin\sqrt{I_{1}}\times
\sqrt{I_{2}}$. Indeed, firstly observe that $I_{1}\neq R_{1}.$ If $I_{1}%
=R_{1},$ then by Theorem \ref{car}, $I_{1}\times I_{2}=R_{1}\times0$ is a
2-absorbing quasi primary, a contradiction. Thus we conclude that
$(a,x)\in\sqrt{I_{1}\times I_{2}}=\sqrt{I_{1}}\times\sqrt{I_{2}}$ or
$(a,y)\in\sqrt{I_{1}}\times\sqrt{I_{2}}.$ This implies $x\in\sqrt{I_{2}}$\ or
$y\in\sqrt{I_{2}}.$ Hence $I_{2}=0$ is quasi primary. Now, let us show that
$I_{1}$ is weaklyquasi primary ideal of $R_{1}.$ Choose $x,y\in R_{1}$ such
that $0\neq xy\in I_{1}.$ Then $0\neq(x,1)(y,1)(1,0)=(xy,0)\in I_{1}%
\times0=I_{1}\times I_{2}.$ As $I_{1}\times I_{2}$ is weakly 2-absorbing quasi
primary and $(xy,1)\notin\sqrt{I_{1}\times0},$ we have $(y,0)\in\sqrt
{I_{1}\times0}$ or $(x,0)\in\sqrt{I_{1}\times0}.$ This implies that $x\in
\sqrt{I_{1}}$\ or $y\in\sqrt{I_{1}}.$ Finally, we show that $I_{1}$ is not
quasi primary. Suppose that $I_{1}$ is quasi primary. As $I_{2}=0$ is a quasi
primary, we have that $I_{1}\times I_{2}$ is 2-absorbing quasi primary by
\cite[Theorem 2.3]{qua}. This contradicts with our assumption. Thus $I_{1}$ is
not quasi primary. Conversely, assume that (i) holds. Let us prove
$I_{1}\times I_{2}$ is weakly 2-absorbing quasi primary. Let $(0,0)\neq
(a_{1},a_{2})(b_{1},b_{2})(c_{1},c_{2})\in I=I_{1}\times I_{2}=I_{1}\times0.$
As $a_{2}b_{2}c_{2}=0,$ we get $a_{1}b_{1}c_{1}\neq0.\ $Since $a_{2}b_{2}%
c_{2}\in I_{2}$ and $I_{2}$ is a quasi primary ideal of $R_{2},\ $we get
either $a_{2}\in\sqrt{I_{2}}\ $or $b_{2}\in\sqrt{I_{2}}$ or $c_{2}\in
\sqrt{I_{2}}.\ $Without loss of generality, we may assume that $a_{2}\in
\sqrt{I_{2}}.\ $On the other hand, since $0\neq a_{1}b_{1}c_{1}=b_{1}%
(a_{1}c_{1})\in I_{1}\ $and $I_{1}\ $is a weakly quasi primary ideal, we have
either $b_{1}\in\sqrt{I_{1}}$ or $a_{1}c_{1}\in\sqrt{I_{1}}.\ $This implies
that either $(a_{1},a_{2})(b_{1},b_{2})\in\sqrt{I_{1}\times I_{2}}$ or
$(a_{1},a_{2})(c_{1},c_{2})\in\sqrt{I_{1}\times I_{2}}$.\ In other cases, one
can similarly show that $(a_{1},a_{2})(b_{1},b_{2})\in\sqrt{I_{1}\times I_{2}%
}$ or $(a_{1},a_{2})(c_{1},c_{2})\in\sqrt{I_{1}\times I_{2}}$ or $(b_{1}%
,b_{2})(c_{1},c_{2})\in\sqrt{I_{1}\times I_{2}}.\ $Hence, $I_{1}\times
I_{2}\ $is weakly 2-absorbing quasi primary ideal of $R.\ $Also, since
$I_{1}\ $is not quasi primary ideal, $I_{1}\times I_{2}$ is not a 2-absorbing
quasi primary ideal by \cite[Theorem 2.3]{qua}.
\end{proof}

\begin{theorem}
\label{car3}Let $R_{1}$ and $R_{2}$ be commutative rings with identity and let
$R=R_{1}\times R_{2}$. Suppose that $\psi_{i}:L(R_{i})\rightarrow L(R_{i}%
)\cup\{\emptyset\}$ $(i=1,2)$ are functions and $\phi=\psi_{1}\times\psi_{2}.$
Let $I=I_{1}\times I_{2}$ be a nonzero ideal of $R$ and $\phi(I)\neq
I_{1}\times I_{2}.$ Then $I_{1}\times I_{2}$ is $\phi$-2-absorbing quasi
primary that is not 2-absorbing quasi primary if and only if $\phi
(I)\neq\emptyset$ and one of the following statements holds.

(i) $\psi_{2}(R_{2})=R_{2}$ and $I_{1}$ is a $\psi_{1}$-2-absorbing quasi
primary ideal of $R_{1}$ that is not a 2-absorbing quasi primary ideal of
$R_{1}.$ $\ \ \ $

(ii)$\ \psi_{1}(R_{1})=R_{1}$ and $I_{2}$ is a $\psi_{2}$-2-absorbing quasi
primary ideal of $R_{2}$ that is not a 2-absorbing quasi primary ideal of
$R_{2}.$

(iii)$\ I_{2}=\psi_{2}(I_{2})$ is a quasi primary of $R_{2}$ and $I_{1}\neq
R_{1}$ is a $\psi_{1}$-quasi primary ideal of $R_{1}$ that is not quasi
primary such that $I_{1}\neq\psi_{1}(I_{1})$ (note that if $I_{1}=0,$ then
$I_{2}\neq0)$

(iv) $I_{1}=\psi_{1}(I_{1})$ is a quasi primary of $R_{1}$ and $I_{2}\neq
R_{2}$ is a $\psi_{2}$-quasi primary ideal of $R_{2}$ that is not quasi
primary such that $I_{2}\neq\psi_{2}(I_{2})$ (note that if $I_{2}=0,$ then
$I_{1}\neq0)$
\end{theorem}

\begin{proof}
Suppose that $I_{1}\times I_{2}$ is a $\phi$-2-absorbing quasi primary ideal
that is not 2-absorbing quasi primary. Then $\phi(I)\neq\emptyset.$ Let
$I_{1}=R_{1}.$ Then $\psi_{1}(R_{1})=R_{1}$ and $I_{2}$ is a $\psi_{2}%
$-2-absorbing quasi primary ideal of $R_{2}$ that is not a 2-absorbing quasi
primary ideal of $R_{2}$ by Theorem \ref{car2}. Let $I_{2}=R_{2}.$ Then
$\psi_{2}(R_{2})=R_{2}$ and $I_{1}$ is a $\psi_{1}$-2-absorbing quasi primary
ideal of $R_{1}$ that is not a 2-absorbing quasi primary ideal of $R_{1}$ by
Theorem \ref{car2}. Hence assume that $I_{1}\neq R_{1}$ and $I_{2}\neq R_{2}.$
Since $\phi(I)\neq I_{1}\times I_{2},$ we obtain that $I/\phi(I)$ is a nonzero
weakly 2-absorbing quasi primary ideal of $R/\phi(I)$ that is not a
2-absorbing quasi primary by Proposition \ref{pq}(ii). Thus $I_{1}/\psi
_{1}(I_{1})\times I_{2}/\psi_{2}(I_{2})$ is a nonzero weakly 2-absorbing quasi
primary ideal of $R_{1}/\psi_{1}(I_{1})\times R_{2}/\psi_{2}(I_{2})$ that is
not a 2-absorbing quasi primary. Then by Theorem \ref{weak3}, we know that one
of the following cases holds:

Case 1: $I_{1}/\psi_{1}(I_{1})=\psi_{1}(I_{1})/\psi_{1}(I_{1})$ is a quasi
primary ideal of $R_{1}/\psi_{1}(I_{1})$ and $I_{2}/\psi_{2}(I_{2})$ is a
non-zero weakly quasi primary ideal of $R_{2}/\psi_{2}(I_{2})$ that is not
quasi primary.

Case 2: $I_{2}/\psi_{2}(I_{2})=\psi_{2}(I_{2})/\psi_{2}(I_{2})$ is a quasi
primary ideal of $R_{2}/\psi_{2}(I_{2})$ and $I_{1}/\psi_{1}(I_{1})$ is a
non-zero weakly quasi primary ideal of $R_{1}/\psi_{1}(I_{1})$ that is not
quasi primary.

Thus, (iii) or (iv) holds by Proposition \ref{pq}(i) and Proposition \ref{pq2}(i).

Conversely, assume that $\phi(I)\neq\emptyset.$ If (i) or (ii) holds, then
$I_{1}\times I_{2}$ is $\phi$-2-absorbing quasi primary that is not
2-absorbing quasi primary by Theorem \ref{car2}. Assume that (iii) or (iv)
holds, then $I/\phi(I)$ is a non-zero weakly 2-absorbing quasi primary ideal
of $R/\phi(I)$ that is not a 2-absorbing quasi primary by Theorem \ref{weak3}.
Thus $I_{1}\times I_{2}$ is $\phi$-2-absorbing quasi primary that is not
2-absorbing quasi primary of $R$ by Proposition \ref{pq}(ii) and Proposition
\ref{pq2}(ii).
\end{proof}

\begin{theorem}
\label{car4}Let $R_{1}$ and $R_{2}$ be commutative rings with identity and
$I_{1}$,$I_{2}$ be nonzero ideals of $R_{1}$ and $R_{2}$, respectively. Let
$R=R_{1}\times R_{2}$ and $\psi_{i}:L(R_{i})\rightarrow L(R_{i})\cup
\{\emptyset\}$ $(i=1,2)$ be functions such that $\psi_{1}(I_{1})\neq I_{1}$
and$\ \psi_{2}(I_{2})\neq I_{2}.$ Suppose that $\phi=\psi_{1}\times\psi_{2}$
and $I_{1}\times I_{2}$ is a proper ideal of $R.$ Then the following
assertions are equivalent:

(i)$\ I_{1}\times I_{2}$ is a $\phi$-2-absorbing quasi primary ideal of $R.$

(ii)\ Either $I_{1}=R_{1}$ and $I_{2}$ is a 2-absorbing quasi primary ideal of
$R_{2}$ or $I_{2}=R_{2}$ and $I_{1}$ is a 2-absorbing quasi primary ideal of
$R_{1}$ or $I_{1}$, $I_{2}$ are quasi primary ideals of $R_{1}$ and $R_{2}$, respectively.

(iii) $I_{1}\times I_{2}$ is a 2-absorbing quasi primary ideal of $R.$
\end{theorem}

\begin{proof}
Assume that $\psi_{1}(I_{1})=\emptyset$ or $\psi_{2}(I_{2})=\emptyset.$ Then
clearly $\phi(I_{1}\times I_{2})=\emptyset$ so that $(i)\Leftrightarrow
(ii)\Leftrightarrow(iii)\ $follows from \cite[Theorem 2.23]{qua}. Hence
suppose that $\psi_{1}(I_{1})\neq\emptyset$ and $\psi_{2}(I_{2})\neq
\emptyset,$ so $\phi(I_{1}\times I_{2})\neq\emptyset.$

$(i)\Rightarrow(ii):$ Let $I_{1}\times I_{2}$ be a $\phi$-2-absorbing quasi
primary ideal of $R.$ Thus $I_{1}/\psi_{1}(I_{1})\times I_{2}/\psi_{2}(I_{2})$
is a non-zero weakly 2-absorbing quasi primary ideal of $R_{1}/\psi_{1}%
(I_{1})\times R_{2}/\psi_{2}(I_{2})$ by Proposition \ref{pq}(ii). Then by
Theorem \ref{weak2}, we know that one of the following cases holds:

\textbf{Case 1}: $I_{1}/\psi_{1}(I_{1})=R_{1}/\psi_{1}(I_{1})$ and $I_{2}%
/\psi_{2}(I_{2})$ is a 2-absorbing quasi primary ideal of $R_{2}/\psi
_{2}(I_{2}).$\ Then we have $I_{1}=R_{1}\ $and $I_{2}$ is a 2-absorbing quasi
primary ideal of $R_{2}.$

\textbf{Case 2}: $I_{2}/\psi_{2}(I_{2})=R_{2}/\psi_{2}(I_{2})$ and $I_{1}%
/\psi_{1}(I_{1})$ is a 2-absorbing quasi primary ideal of $R_{1}/\psi
_{1}(I_{1})$.\ Similar to Case 1, $I_{2}=R_{2}\ $and $I_{1}$ is a 2-absorbing
quasi primary ideal of $R_{1}.$

\textbf{Case 3}: $I_{1}/\psi_{1}(I_{1})$ and $I_{2}/\psi_{2}(I_{2})$ are quasi
primary of $R_{1}/\psi_{1}(I_{1}),R_{2}/\psi_{2}(I_{2}),$ respectively. Then
$I_{1},I_{2}\ $are quasi primary ideals of $R_{1}\ $and $R_{2},\ $%
respectively\ by Proposition \ref{pq2}(ii).

$(ii)\Rightarrow(iii):$ Assume that $I_{1}=R_{1}$ and $I_{2}$ is a 2-absorbing
quasi primary ideal of $R_{2}$ or $I_{2}=R_{2}$ and $I_{1}$ is a 2-absorbing
quasi primary ideal of $R_{1}$ or $I_{1}$, $I_{2}$ are quasi primary ideals of
$R_{1}$ and $R_{2}$, respectively. Then by Theorem Theorem \cite[Theorem
2.23]{qua}, $I_{1}\times I_{2}$ is a 2-absorbing quasi primary ideal of $R.$

$(iii)\Rightarrow(i):$ It is evident.
\end{proof}

\end{document}